\documentclass[draft]{amsart}

\usepackage{amssymb}
\usepackage{url}
\usepackage{amscd}

\theoremstyle{plain}
\newtheorem{theorem}{Theorem}

\newtheorem{corollary}[theorem]{Corollary}
\newtheorem{lemma}[theorem]{Lemma}

\newtheorem{definition}[theorem]{Definition}
\theoremstyle{remark}
\newtheorem{remark}{Remark}
\newtheorem{example}{Example}

\begin{document}

\date{}

\title{On $T-$locally compact spaces}

\author{Aliakbar Alijani}

\address{Mollasadra Technical and Vocational College\\
Technical and Vocational University\\
Ramsar, Iran}

\email{alijanialiakbar@gmail.com}

\thanks{}


\begin{abstract}
The aim of this paper is to introduce and give preliminary investigation of $T-$locally compact spaces. Locally compact and $T-$locally compact are independent of each other. Every Hausdorff, locally compact space is $T-$locally compact. $T-$locally compact is a topological property. $T-$locally compact is not preserved by the product topology.

\end{abstract}

\maketitle

%%%%%%%%%%%%%%%%%%%%%%%%%%%%%%%%%%%%%%%%%%%%%%%%%%%%%%%%%%%%%%%%%%%%%%%%%
% Macros
%%%%%%%%%%%%%%%%%%%%%%%%%%%%%%%%%%%%%%%%%%%%%%%%%%%%%%%%%%%%%%%%%%%%%%%%%

\newcommand\sfrac[2]{{#1/#2}}

\newcommand\cont{\operatorname{cont}}
\newcommand\diff{\operatorname{diff}}

%%%%%%%%%%%%%%%%%%%%%%%%%%%%%%%%%%%%%%%%%%%%%%%%%%%%%%%%%%%%%%%%%%%%%%%%%

\section{Introduction}
By a space, we mean a topological space. A space $X$ will be called $T-$locally compact if for every open set $U$ containing $x$, there exists an open subset $V$ containing $x$ such that $\partial V$ is compact and $V\subseteq U$. The space $X$ is called $T-$locally compact if $X$ is $T-$locally compact at each of points. Locally compact and $T-$locally compact are independent of each other (Examples \ref{2} and Example \ref{3}). We show that a Hausdorff, locally compact space is $T-$locally compact (Lemma \ref{4}). Open or closed subspace of a $T-$locally compact space is $T-$locally compact (Lemma \ref{10} and Lemma \ref{6}). We show that $T-$locally compact is a topological property (Theorem \ref{9}). The product of two $T-$locally compact spaces need not to be $T-$locally compact (Example \ref{7}). We give some conditions such that the product of two $T-$locally compact spaces is $T-$locally compact( Lemma \ref{11} and Theorem \ref{8}).

Throughout, $cl A$, $int A$ and $\partial A$ will denote the closure, the interior and the boundary of a set $A$ respectively. Assume that $I$ be a non-empty index set and for every $i\in I$, $X_{i}$ be a space. We denote by $\prod_{i\in I}X_{i}$, the cartesian product of $X_{i}$ with the product topology. For more information on topological spaces, see \cite{B}.

%%%%%%%%%%%%%%%

\section{$T-$locally compact space}
In this section, we introduce the concept and study some properties of $T-$locally compact space.

\begin{definition}
A space $X$ is called $T-$locally compact at $x$ if for every open set $U$ containing $x$, there exists an open subset $V$ containing $x$ such that $\partial V$ is compact and $V\subseteq U$. The space $X$ is called $T-$locally compact if $X$ is $T-$locally compact at each of points.
\end{definition}

\begin{example}
Let $\Bbb R$ be the reals with the usual topology. Assume that $x\in U$ for some open subset $U$ of $\Bbb R$. Then, $x\in (a,b)$ for some real numbers $a$ and $b$ which $(a,b)\subseteq U$. It is clear that $\partial(a,b)=\{a,b\}$ is compact in $\Bbb R$. So, $\Bbb R$ is $T-$ locally compact. Also, it is clear that every discrete space is $T-$locally compact.
\end{example}

\begin{definition}
A subset $A$ of a space $X$ is called nowhere dense if $X-cl A$ is dense.
\end{definition}

\begin{lemma}\label{1}
The boundary of an open or closed subsets of a space is nowhere dense.
\end{lemma}

\begin{proof}
It is clear.
\end{proof}

Locally compact and $T-$locally compact are independent of each other. See the examples \ref{2} and \ref{3}.

\begin{example}\label{2}
Consider $\Bbb Q$ as a subspace of $\Bbb R$ (with the usual topology) and $U$, be an open subset of $\Bbb Q$. We know that the only compact sets in $\Bbb Q$ are nowhere dense. Hence, by Lemma \ref{1}, $\partial U$ is compact. So, $\Bbb Q$ is $T-$locally compact. But, $\Bbb Q$ is not locally compact.
\end{example}

\begin{example}\label{3}
Let $X$ be an infinite set and $p\in X$. Define $\tau={\phi}\bigcup \{U;p\in U\}$. Then, $\tau$ is a topology on $X$. It is clear that $(X,\tau)$ is not Hausdorff, locally compact space. Since $\partial \{p\}=X-\{p\}$ is not compact, $(X,\tau)$ is not a $T-$locally compact space.
\end{example}

\begin{lemma} \label{4}
A Hausdorff, locally compact space is a $T-$locally compact space.
\end{lemma}

\begin{proof}
Let $x\in X$ and $U$ be an open set containing $x$. Since $X$ is Hausdorff and locally compact, there is an open set $V$ containing $x$ such that $cl V$ is compact and $V\subseteq U$. It is clear that $\partial V$ is compact as a closed subset of compact set $cl V$.
\end{proof}

\begin{lemma} \label{5}
Every compact space $X$ is $T-$locally compact, even though $X$ is not Hausdorff.
\end{lemma}

\begin{proof}
Let $X$ be a compact space. Then, for every open subset $U$ of $X$, $\partial U$ is compact and proof is complete.
\end{proof}

\begin{lemma}\label{10}
An open subspace of a $T-$locally compact space is $T-$locally compact.
\end{lemma}

\begin{proof}
Let $X$ be a $T-$locally compact space and $Y$, an open subspace of $X$. Let $y\in Y$ and $U$ be an open set in $Y$ containing $y$. Then, $U$ is open in $X$. Since $X$ is $T-$locally compact, there is an open set $V$ in $X$ containing $y$ such that $\partial V$ is compact and $ V\subseteq U$. So, $Y$ os $T-$locally compact.
\end{proof}

\begin{remark}\label{12}
Let $Y$ be a closed subspace of $X$ and $U$, an open subset in $X$. Then,
\begin{align*}
 \partial_{Y} (U\cap Y)&=cl_{Y}(U\cap Y)-int_{Y}(U\cap Y)\\ \nonumber
 &=cl U\cap Y-U\cap Y\\\nonumber
 &=\partial U\cap Y
\end{align*}
\end{remark}

\begin{lemma} \label{6}
A closed subspace of a $T-$locally compact space is $T-$locally compact.
\end{lemma}

\begin{proof}
Let $X$ be a $T-$locally compact space and $Y$, a closed subspace of $X$. Let $y\in Y$ and $U$ be an open set in $Y$ containing $y$. Then, there exists an open set $W$ in $X$ such that $U=W\cap Y$. Since $X$ is $T-$locally compact, there exists an open set $V$ containing $y$ such that $\partial V$ is compact and $V\subseteq W$. By Remark \ref{12}, $\partial_{Y} (V\cap Y)=\partial V\cap Y$ which is compact. Also, $V\cap Y\subseteq U$. So, $Y$ is $T-$locally compact.
\end{proof}

\begin{theorem}\label{9}
$T-$locally compact is a topological property.
\end{theorem}

\begin{proof}
Let $X$ and $Y$ be two spaces and $f:X\to Y$, a homomorphism. Let $X$ be a $T-$locally compact space. We show that $Y$ is $T-$locally compact. Let $y\in Y$ and $V$ be an open subset of $Y$ containing $y$. Then, $f(x)=y$ for some $x\in X$ and $x\in f^{-1}(V)$ is an open subset of $X$. There exists an open subset $U$ of $X$ containing $x$ such that $U\subseteq f^{-1}(V)$ and $\partial U$ is compact. Since $f$ is open, $f(U)$ is open in $Y$ and is contained in $V$. Also, $$\partial f(U)=cl f(U)-f(U)=f(cl U)-f(U)\subseteq f(\partial U)$$
So, $\partial f(U)$ is compact and $Y$ is $T-$locally compact.
\end{proof}

If $X$ and $Y$ be two $T-$locally compact spaces, then, $X\times Y$ need not to be $T-$locally compact. See the Example \ref{7}.

\begin{example}\label{7}
Let $\Bbb R$ be the reals with the usual topology and $\Bbb Q$, the rationales with the subspace topology. We show that $\Bbb Q\times \Bbb R$ is not $T-$locally compact. Let $N=(0,1)^{2}\cap (\Bbb Q\times \Bbb R)$ ($(0,1)^{2}=(0,1)\times (0,1)$). We claim that the boundary of every nonempty open subset of $N$ is not compact. Let $U\subseteq N$ be an open set. First, we show that $\pi_{1} (U)\subseteq\pi_{1}(\partial U)$. Let $x\in \pi_{1} U$. Assume to contrary, $x\notin \pi_{1}(\partial U)$. Then, $\{x\}\times \Bbb R\subseteq U\subseteq N$. So, $\pi_{2}(\{x\}\times \Bbb R)\subseteq \pi_{2}(N)=(0,1)$ which is a contradiction. Now, if $\partial U$ is compact, then $cl \pi_{1} (U)$ is compact in $\Bbb Q$ which is a contradiction (since $int \pi_{1} (U)\neq \emptyset$ ). So, $\Bbb Q\times \Bbb R$ is not $T-$locally compact.
\end{example}

\begin{lemma}\label{11}
Let $X$ be a discrete space and $Y$, a $T-$locally compact space. Then, $X\times Y$ is $T-$locally compact.
\end{lemma}

\begin{proof}
Let $N$ be an open subset of $X\times Y$ containing $(x,y)$. Then, there exists an open subset $V$ of $Y$ containing $y$ such that $\partial V$ is compact. It is clear that $(\{x\}\times V )\cap N$ is an open subset of $X\times Y$ containing $(x,y)$. Also, $$\partial ((\{x\}\times V )\cap N)\subseteq (\{x\}\times \partial V)\cap \partial N$$
So, $\partial ((\{x\}\times V )\cap N)$ is compact. Hence, $X\times Y$ is $T-$locally compact.
\end{proof}

\begin{theorem}\label{8}
Let $Y$ be a compact space. Then, $X\times Y$ is $T-$locally compact if and only if $X$ is $T-$locally compact.
\end{theorem}

\begin{proof}
First, suppose that $X\times Y$ be $T-$locally compact. Let $U$ be an open subset of $X$ containing $x$. The, $U\times Y$ is an open set in $X\times Y$ containing $(x,y)$ for some $y\in Y$. So, there exists an open subset $N$ of $X\times Y$ containing $(x,y)$ such that $\partial N$ is compact and $N\subseteq X\times Y$. Since $Y$ is compact, $\pi_{1}$ is a closed map. Hence, $\partial \pi_{1}(N)\subseteq \pi_{1}(\partial N)$. So, $\partial \pi_{1}(N)$ is compact and $\pi_{1}(N)\subseteq U$. Conversely, Let $N$ be an open subset of $X\times Y$ containing $(x,y)$. Then, $\pi_{1}(N)$ is an open set in $X$ containing $x$. Since $X$ is $T-$locally compact, there exists an open set $U$ of $X$ containing $x$ such that $\partial U$ is compact. Clearly, $(x,y)\in (U\times Y)\cap N\neq \emptyset$. Since $\partial ((U\times Y)\cap N)\subseteq (\partial U\times Y)\cap \partial N$, so $(U\times Y)\cap N$ is an open set containing $(x,y)$ such that $\partial ((U\times Y)\cap N)$ is compact. It shows that $X\times Y$ is $T-$locally compact.
\end{proof}

\begin{corollary}
Let $X$ and $Y$ be two Hausdorff spaces. If $X\times Y$ is $T-$locally compact, then $X$ and $Y$ are $T-$locally compact.
\end{corollary}

\begin{proof}
Let $y\in Y$. By Lemma \ref{6}, $X\times \{y\}$ is $T-$locally compact. Hence, by Theorem \ref{8}, $X$ is $T-$locally compact. Similarly, $Y$ is $T-$locally compact.
\end{proof}

\begin{corollary}
Let $\{X_{i};i\in I\}$ be an arbitrary family of Hausdorff spaces. If $\prod_{i}X_{i}$ is a $T-$locally compact space, then each $X_{i}$ is $T-$locally compact.
\end{corollary}

\end{document}